\documentclass[11pt]{amsart}

\usepackage{amscd}
\usepackage{amsmath,latexsym,amssymb,amsthm}
\usepackage{hyperref}
\usepackage{lscape,fancyhdr}
\usepackage[top=2.5cm, left=3cm, right=3cm]{geometry}
\usepackage{epsfig,graphicx}
\usepackage{float}
\usepackage{amsfonts}
\newtheorem{secn}{Definition}[section]
\newtheorem{thm}[secn]{Theorem}
\newtheorem{cor}[secn]{Corollary}
\newtheorem{prop}[secn]{Proposition}
\newtheorem{lem}[secn]{Lemma}
\newtheorem{defn}[secn]{Definition}
\newtheorem{ex}[secn]{Example}

\newcommand{\tX}{\tilde{X}}

\begin{document}
\title{Strongly Contracting Geodesics in a tree of spaces}
\author{Abhijit Pal}
\thanks{Research of the first author was supported by INSPIRE Research Grant}
\address{Department of Mathematics and Statistics, Indian Institute of Technology-Kanpur}
\email{abhipal@iitk.ac.in}
\author{Suman Paul}
\address{Department of Mathematics, Indian Institute of Science Education and Research, Bhopal}
\email{smnpl2009@gmail.com}
\date{}
\maketitle
\begin{abstract}
Let $X$ be a tree of proper geodesic spaces with edge spaces strongly contracting and uniformly separated from each other by a number depending on the contraction function
of edge spaces. Then we prove that the strongly contracting geodesics in vertex spaces are quasiconvex in $X$. We further prove that in $X$ if all the vertex spaces are uniformly hyperbolic metric spaces then  $X$ is a hyperbolic metric space and vertex spaces are quasiconvex in $X$.
\end{abstract}
\begin{center}{AMS Subject Classification : 20F65, 20F67, 20E08}\end{center}
\section{Introduction}
Let $X$ be a  geodesic metric space. A subset $Q$ of a geodesic metric space $X$ is said to be \textit{strongly contracting} if there exists a constant $\sigma\geq 0$ such that the diameter of the  projection of balls, disjoint from $Q$, into $Q$ is at most $\sigma$.   For example, the geodesics and horoballs of upper half space with Poincar\'{e} metric are strongly contracting. Using strongly contracting geodesics, Charney-Sultan \cite{Charney15} gave the notion of ‘contracting' boundary for a CAT$(0)$ space which is an analogue of the Gromov boundary. The concept of contracting projection has created a lot of interest in recent times, we refer the reader to look into the work of
Arzhantseva, Cashen, Gruber, and Hume in \cite{ACGH} for more details about the other contraction properties of subspaces.
Given a finite collection of groups satisfying a property $P$, the combination type problems deal with the question that  then when does the fundamental group of the graph  of groups formed by that collection satisfy the property $P$. One such case is the celebrated work of M.Bestvina and M.Feighn \cite{best} in 1992, where they have found a combination theorem of finite graphs of hyperbolic groups. Given a graph of groups $(\mathcal G,\Lambda)$ over a finite graph $\Lambda$, there exists a tree of spaces $X$ where the underlying tree is the Bass-Serre covering tree such that  the fundamental group $\pi_1(\mathcal G,\Lambda)$ of $(\mathcal G,\Lambda)$ acts properly and co-compactly on $X$. An action of a group $G$ on a simplicial tree $T$ is called \textit{acylindrical} if there exists $k\geq 1$ such that no non-trivial element of $G$ fixes point-wise a segment of length $k$ in $T$ . 
I.Kapovich in \cite{kapo} showed that  for a
finite graph of hyperbolic groups $(\mathcal G,\Lambda)$ with edge groups quasi-isometrically embedded in vertex groups if the fundamental group $\pi_1(\mathcal G,\Lambda)$
acts acylindrically on the Bass-Serre covering tree  then $\pi_1(\mathcal G,\Lambda)$ is hyperbolic and all vertex groups are quasiconvex in the $\pi_1(\mathcal G,\Lambda)$.
In this article we prove the following: \\ \\
\textbf{Theorem 1} \label{main intro}(See Theorem \ref{main})\textit{
Let $N\geq 1$ and $\sigma\geq 0$. There exists
$R=R(N,\sigma)\geq 1$ such that the following holds:
Let $X$ be a tree of proper geodesic spaces where all  edge spaces are
$\sigma$-contracting and uniformly $R$-separated. Let $\eta$ be a strongly contracting geodesic in a vertex space.
Then  every $(N,0)$-quasi-geodesic of $X$ with end points in $\eta$ lies in a bounded
neighbourhood of $\eta$ where the bound depends only on $N,\sigma$ and the contraction constant of $\eta$.}\\ \\
A subspace $Y$ of $X$ is said to be $k$-\textit{quasiconvex} if  all geodesics of $X$ with endpoints in $Y$ lie in a $k$-neighbourhood of $Y$.
By taking $N=1$ in Theorem \ref{main intro}, we have the following corollary:\\ \\
\textbf{Corollary :}\textit{
For all $\sigma\geq 0$ there exists  a number $R\geq 1$ depending only on $\sigma$ such that the following holds:
Let $X$ be a tree of proper geodesic spaces whose  edge spaces are $\sigma$-contracting and uniformly $R$-separated. 
Then every strongly contracting geodesic $\gamma$ of a vertex space is quasiconvex in $X$, where the quasiconvexity constant depends only on $\sigma$ and the contraction constant of $\gamma$.}\\ \\
As an application of Theorem \ref{main intro}, we will prove the following theorem in the last section. \\ \\
\textbf{Theorem 2}
(See Theorem \ref{hyp thm})\textit{
Let $X$ be a tree of proper geodesic spaces.
If all the vertex spaces are uniformly hyperbolic metric spaces, edge spaces are uniformly quasiconvex in vertex spaces
and they are sufficiently uniformly separated then the space $X$ is a hyperbolic metric space and vertex spaces are quasiconvex in $X$.}\\ \\
\noindent\textbf{Acknowledgements:} We are thankful to the anonymous referee for his/her valuable comments.

\section{Strong Contraction}

Let $(X,d)$ be a geodesic metric space and $Q$ be a subspace of $X$. The nearest point projection of a point $x\in X$ to $Q$ is given by
$\pi_Q(x)=\{z\in Q:d(x,z)=d(x,Q)\}$. If $Q$ is closed and $X$ is proper then $\pi_Q(x)$ is always non-empty.
\begin{defn}(Strongly Contracting Subspace) Let $\sigma\geq 0$. We say that $Q$ is $\sigma$-contracting if for all $x,y\in X$,
$$d(x,y)\leq d(x,Q)\implies\mbox{diam}(\pi_Q(x)\cup\pi_Q(y))\leq \sigma.$$ 
We say that $Q$ is strongly contracting if $Q$ is $\sigma$-contracting for some $\sigma\geq 0$. 
\end{defn}

\begin{thm}{( Geodesic Image Theorem \cite{ACGH}, Theorem 7.1)}\label{GProj}
 $Q$ is a $\sigma$-contracting subspace of a geodesic space $X$ if and only if there exists constants $\kappa_{\sigma}'$ and $\kappa_{\sigma}$
 such that for any geodesic $\gamma$ in $X$ lying outside $\kappa_{\sigma}$-neighborhood
 of $Q$, the diameter of $\pi_Q(\gamma)$ is at most $\kappa_{\sigma}'$, where $\kappa_{\sigma}'$ and $\kappa_{\sigma}$ depend only on $\sigma$.
\end{thm}

For the sake of completeness, we give here a proof of a quasification of Geodesic Image Theorem \ref{GProj}. 
The proof is adaptation of the arguments
used in \cite{ACGH}.
 
\begin{thm}(Bounded Quasi-geodesic Image)\label{BQI} Let $Q$ be a $\sigma$-contracting
subspace of a geodesic space $X$. Suppose
$\beta:[0,l]\to X$ is a continuous $(K,0)$-quasi-geodesic lying outside
$D$-neighborhood of $Q$, where
$D=D(K,\sigma)=2([K]+1)\sigma$, $l=l(\beta)$ is the length of $\beta$ and $[0,l]$ is the arc
length parameterization of $\beta$ . 
Then diam($\pi_Q(\beta(0))\cup\pi_Q(\beta(l))< 4D$. 
\end{thm}
\begin{proof}
 Let $\beta(0)=x$ and $\beta(l)=y$.\\
 Case(i): $l(\beta)\leq K(d(x,Q)+d(y,Q))$. \\
 It imply either $l(\beta)\leq \frac{K}{2}d(x,Q)$ or $l(\beta)\leq
\frac{K}{2}d(y,Q)$. \\
 Suppose \begin{eqnarray}
          l(\beta)\leq \frac{K}{2}d(x,Q)
          \label{1}
         \end{eqnarray}

 Let us first assume the distance between $\beta$ and $Q$ is realised at $x$ i.e.
 $d(\beta([0,l]),Q)=d(x,Q)$. Let $0=t_0<t_1<...<t_n=l$ be a partition  of $[0,l]$
such that $l(\beta|_{[t_{i},t_{i+1}]})=d(x,Q)$ for all $0\leq i\leq n-2$
 and $l(\beta|_{[t_{n-1},t_n]})\leq d(x,Q)$. Note that $n$ is at most
$[\frac{K}{2}]+1$. For each $i$, 
 $d(\beta(t_i),\beta(t_{i+1}))\leq l(\beta|_{[t_{i},t_{i+1}]})=d(x,Q)\leq
d(\beta(t_i),Q)$. $Q$ is $\sigma$-contracting implies
 $\mbox{diam }(\pi_Q(\beta(t_i))\cup\pi_Q(\beta(t_{i+1})))\leq \sigma$ for all $i$.
 $$\mbox{diam}~(\pi_Q(x))\cup\pi_Q(y))\leq \sum\limits_{i=0}^{n-1}\mbox{diam
}(\pi_Q(\beta(t_i))\cup\pi_Q(\beta(t_{i+1})))\leq ([\frac{K}{2}]+1)\sigma.$$
 If  distance between $\beta$ and $Q$ is not realised at $x$ then let $s\in[0,l]$
such that $d(\beta([0,l]),Q)=d(\beta(s),Q)$. Divide $\beta$ in two parts
 $\beta'$ and $\beta''$ such that $\beta'=\beta|_{[0,s]}$ and
$\beta''=\beta|_{[s,l]}$. After suitable re-parametrization of $\beta'$ and
$\beta''$ 
 we have that distances $d(\beta',Q), d(\beta'',Q)$ are realized at the starting
points of $\beta',\beta''$ respectively. As above, union of nearest point
 projections of end points of $\beta',\beta''$ have diameters at most
$[\frac{K}{2}]+1$. So,  $l(\beta)\leq \frac{K}{2}d(x,Q)$
 implies $\mbox{diam}~(\pi_Q(x))\cup\pi_Q(y))\leq 2([\frac{K}{2}]+1)$. For
$l(\beta)\leq \frac{K}{2}d(y,Q)$ we take the re-parametrization
 $\bar\beta(t)=\beta(l-t)$, $t\in[0,l]$, and repeat the same argument for $\bar\beta$
as above.\\
 Case(ii) : $l(\beta)>K(d(x,Q)+d(y,Q))$. \\
 Let $t_1\in [0,l]$ be such that $l(\beta|_{[0,t_1]})=Kd(x,Q)$. Then our assumption
forces $l(\beta|_{[t_1,l]})>Kd(y,Q)$.
 By replacing $\frac{K}{2}$ with $K$ in the inequality \eqref{1} and repeating the
same argument as in case (i), 
 we have $\mbox{diam }(\pi_Q(\beta(0))\cup\pi_Q(\beta(t_{1})))\leq 2([K]+1)\sigma$. 
We will define inductively
 the points $t_i$ in $[0,l]$ such that they form a partition of $[0,l]$ and
$l(\beta|_{[t_{i-1},t_i]})=Kd(\beta(t_{i-1},Q))$. 
 Suppose we have defined the points $t_0,t_1,...,t_i$. If $l(\beta|_{[t_i,l]})\leq
K(d(\beta(t_i),Q)+d(y,Q))$ then by case (i)
 $\mbox{diam }(\pi_Q(\beta(t_i))\cup\pi_Q(\beta(l)))\leq 2([\frac{K}{2}]+1)\sigma$ and
we define $t_{i+1}=l$.
 Otherwise if $l(\beta|_{[t_i,l]})>K(d(\beta(t_i),Q)+d(y,Q))$ define
$t_{i+1}\in[t_i,l]$ to be the point such that
 $l(\beta|_{[t_i,t_{i+1}]})=Kd(\beta(t_i),Q)$. Again this will imply
$l(\beta|_{[t_{i+1},l]})>Kd(y,Q)$. Let $t_r$ be the last index
 for which $l(\beta|_{[t_r,l]})>K(d(\beta(t_r),Q)+d(y,Q))$ and define $t_{r+1}=l$.
Let $D=2([K]+1)\sigma$ and 
 $\beta_i=\beta|_{[t_i,t_{i+1}]}$ for all $i\in\{0,1,...,r\}$.
 \begin{eqnarray}
  l(\beta)&=&\sum\limits_{i=0}^rl(\beta_i)=\sum\limits_{i=0}^{r-1}Kd(\beta(t_i),Q)+l(\beta_r)\\
  \label{2} l(\beta)&=& Kd(x,Q)+\sum\limits_{i=1}^{r-1}Kd(\beta(t_i),Q)+l(\beta_r)
 \end{eqnarray}
Also, $\beta$ being $(K,0)$-quasi-geodesic, we have
\begin{eqnarray}
 l(\beta)&\leq& Kd(x,y)\\
\label{3} l(\beta)&\leq& Kd(x,\pi_Q(x))+K\sum\limits_{i=0}^r\mbox{diam
}(\pi_Q(\beta(t_i))\cup\pi_Q(\beta(t_{i+1})))+Kd(\pi_Q(y),y)
\end{eqnarray}
From \eqref{2}\& \eqref{3} and as $d(x,Q)=d(x,\pi_Q(x))$, we have
\begin{eqnarray*}
 \sum\limits_{i=1}^{r-1}Kd(\beta(t_i),Q)+l(\beta_r)&\leq&
 K\sum\limits_{i=0}^r\mbox{diam }(\pi_Q(\beta(t_i))\cup\pi_Q(\beta(t_{i+1})))+Kd(y,Q)\\
 &\leq& 2KD+(r-1)KD+Kd(y,Q)
\end{eqnarray*}

Thus, \begin{eqnarray}\label{4}
       \sum\limits_{i=1}^{r-1}K(d(\beta(t_i),Q)-D)\leq 2KD-(l(\beta_r)-Kd(y,Q)).
      \end{eqnarray}
By construction of the point $t_r$, $l(\beta_r)-Kd(y,Q)>0$, so from \eqref{4} we have 
\begin{eqnarray}\label{5}
 \sum\limits_{i=1}^{r-1}(d(\beta(t_i),Q)-D)<2D
\end{eqnarray}
As per hypothesis, $\beta$ lies outside $2D$-neighborhood of $Q$, so $d(\beta,Q)>2D$.
This implies
$d(\beta(t_i),Q)>2D$ and hence $D<d(\beta(t_i),Q)-D$.
So, from \eqref{5}, we have
\begin{eqnarray*}
 \mbox{diam}(\pi_Q(x)\cup\pi_Q(y))&\leq&
\sum\limits_{i=0}^r\mbox{diam}(\pi_Q(\beta(t_i))\cup\pi_Q(\beta(t_{i+1})))\\
 &\leq& (r+1)D\\
 &<& \sum\limits_{i=1}^{r-1}(d(\beta(t_i),Q)-D) + 2D\\
 &<& 2D+2D=4D.
\end{eqnarray*}
\end{proof}
\noindent\textbf{Notation :} For a path $\alpha:[a,b]\to X$  with $p=\alpha(s)$, $q=\alpha(t)$ and $s<t$, we denote $\alpha|_p^q$ to be the subsegment of $\alpha$ between $p$ and $q$. We denote $\ell_X(\alpha|_p^q)$ to be the length of $\alpha|_p^q$ in $X$.
 
\begin{lem}\label{closeness}
Let $Q$ be $\sigma$-contracting and properly embedded in a geodesic metric space $(X,d)$. Then \\
(i) any $(N,0)$-quasi-geodesic between two points of $Q$ lies in a $M=M(N,\sigma)$-neighbourhood of $Q$, 
where $M(N,\sigma)=(3N+1)D(N,\sigma)$ and $D=D(N,\sigma)$ is the constant from Theorem \ref{BQI}, depends only on $N$ and $\sigma$.\\
(ii) there exists $K\geq 1$ depending only on $\sigma$ such that $Q$ is $(K,K)$-quasi-isometrically embedded in $X$.

\end{lem}
\begin{proof}
(i) Let $x,y\in Q$ and $\gamma$ be $(N,0)$-quasi-geodesic between $x$ and $y$. Let $\gamma|_{p}^{q}$ be a maximal connected subsegment of $\gamma$ which lie outside 
 $D$ neighbourhood of $Q$, then 
 \begin{center}
    $d(p,Q)=D=d(q,Q)$ and $diam(\pi_{Q}(p),\pi_{Q}(q))<4D$  (By Theorem \ref{BQI})
 \end{center}
 So $\ell_{X}(\gamma|_{p}^{q})\leq Nd(p,q)<N(2D+4D)=6ND$. 
 Now for any $z\in \gamma|_{p}^{q}$ then either $\ell_{X}(\gamma|_{p}^{z})<3ND$ or $\ell_{X}(\gamma|_{z}^{q})<3ND$ and therefore $d(z,Q)<3ND+D=(3N+1)D$. This holds for any maximal connected subsegment of $\gamma$ which lie outside $D$-neighbourhood of $Q$. Hence $\gamma$ lies in the $M$-neighbourhood of $Q$.
 \\
 (ii)  Let $x,y\in Q$ and $\gamma:[0,l]\to X$ be a $X$-geodesic with $\gamma(0)=x$ and $\gamma(l)=y$. Let $x_i=\gamma(i)$ where $i\in\{1,...,[l]\}$.
 Then $d(x,x_1)=1,d(x_i,x_{i+1})=1,...,d(x_{[l]},x_l)\leq 1$. From \textit{(i)}, $\gamma$ lies in $4D(1,\sigma)$-neighbourhood of $Q$, as geodesics are $(1,0)$-quasi-geodesics. Hence, $d(x_i,\pi_Q(x_i))\leq 4D(1,\sigma)$. This implies, $d(\pi_Q(x_i),\pi_Q(x_{i+1}))\leq 8D(1,\sigma)+1$. As $Q$ is properly embedded in $X$, there exists $K=K(D(1,\sigma))\geq 1$ such that $$d_Q(\pi_Q(x_i),\pi_Q(x_{i+1}))\leq K.$$
 Now consecutively joining $\pi_Q(x_i)$'s by a $Q$-geodesic we obtain a path $\Tilde{\gamma}$ between $x$ and $y$ in $Q$.
 And we have $d_Q(x,y)\leq \ell_{X}(\Tilde{\gamma}) \leq [l]K+K\leq Kd(x,y)+K.$

\end{proof}
\begin{lem}\label{Proj}
Let $\sigma\geq 0$. Let $C_{1}:=5D(1,\sigma)$ and $D_{1}:=8D(1,\sigma)$, where $D(1,\sigma)$ is the constant from Theorem \ref{BQI} for $N=1$. Let $Q_1$ and $Q_2$ be $\sigma$-contracting in $X$ then
 \begin{center}
     $d(Q_1,Q_2)>C_{1} \implies diam(\pi_{Q_1}(Q_2))<D_{1}$.
 \end{center}
\end{lem}
\begin{proof}
 Fix a point $x_{0}\in Q_2$. Let $x$ be an arbitrary point in $Q_2$ and $\gamma$ be a $X$-geodesic between $x_{0}$ and $x$. 
 Then by Lemma \ref{closeness}, $\gamma$ lies in $4D(1,\sigma)$-neighbourhood of $Q_2$ (Note that geodesics are $(1,0)$-quasi-geodesics).
 
 Now $d(Q_1,Q_2)>C_{1}=5D(1,\sigma)$ implies $d(\gamma,Q_1)>D(1,\sigma)$. Then by Theorem \ref{BQI}
 \begin{center}
     $diam(\pi_{Q_1}(x_0)\cup\pi_{Q_1}(x))<4D(1,\sigma)$.
 \end{center}
 Hence $diam(\pi_{Q_1}(Q_2))<2\times 4D(1,\sigma)=8D(1,\sigma)=D_1$.
\end{proof}

\subsection{Tree of Spaces}

\begin{defn} \label{tree}(1) (Tree of Geodesic Spaces:) Let $(X,d_X)$ be a proper geodesic metric space.
$P: X \rightarrow T$ is said to be a  tree of geodesic
 metric spaces if $X$
admits a surjective map $P : X \rightarrow T$ onto a simplicial tree $T$, such that the following holds:\\
i) For all  $s\in{T}$, 
$X_s= P^{-1}(s) \subset X$ with the induced path metric $d_{X_s}$ is 
 a geodesic metric space $X_s$.\\
(ii) For a vertex $v$ in $T$, $X_v=P^{-1}(v)$ will be called as vertex space for $v$.  Let $e$ be an edge of $T$ between vertices $v_1$ and
$v_2$.
Let $X_e$ be the preimage under $P$ of the mid-point of  $e$, $X_e$ will be called as edge space for $e$.
 There exist a continuous map ${f_e}:{X_e}{\times}[0,1]\rightarrow{X}$, such that
$f_e{|}_{{X_e}{\times}(0,1)}$ is an isometry onto the preimage of the
interior of $e$ equipped with the path metric. 
The maps ${f_e}|_{{X_e}{\times}\{{0}\}}$ and 
${f_e}|_{{X_e}{\times}\{{1}\}}$ are proper
embeddings into $X_{v_1}$ and $X_{v_2}$ respectively. \\ \\
(2) (Edge Spaces Separation:) Let $P:X\to T$ be a tree of geodesic metric spaces. For  an edge $e$ of $T$ between vertices $v_1$ and $v_2$, let $f_{e,v_1}:X_e\to X_{v_1}$ be defined by $f_{e,v_1}(x)=f_e(x,0)$ and $f_{e,v_2}:X_e\to X_{v_2}$ be defined by $f_{e,v_2}(x)=f_e(x,1)$.
The edge spaces in  $P:X\to T$ are said to be $R$-\textit{uniformly separated}
if for  any two distinct edges $e,e^{'}$ of $T$ with a common vertex $v$ we have 
$d_X(f_{e,v}(X_e),f_{e',v}(X_{e'})\geq R$.

\end{defn}
\begin{defn} (Strongly contracting edge spaces) Let $\sigma\geq 0$. We say that the edge spaces of a tree of spaces $P:X\to T$
are $\sigma $-contracting in vertex spaces if for all edge spaces $X_e$, ${f_e}({{X_e}{\times}\{{0}\}})$ and 
${f_e}({{X_e}{\times}\{{1}\}})$ are $\sigma$-contracting in respective vertex spaces.
 
\end{defn}

\section {Main Theorem}

Let $P:X\to T$ be a tree of proper geodesic spaces.
The edge space $X_e$ corresponding to an edge $e$ defined in Definition \ref{tree} are identified to respective end vertex spaces with the help of maps ${f_e}|_{{X_e}{\times}\{{0}\}}$ and 
${f_e}|_{{X_e}{\times}\{{1}\}}$, so we will call ${f_e}({{X_e}{\times}\{{0}\}})$ and 
${f_e}({{X_e}{\times}\{{1}\}})$ also to be edge spaces.
For an  edge $e$, we denote $f_e(X_e\times\{0\})$ by $Y^-_e$ and $f_e(X_e\times\{1\})$ by $Y^+_e$.

\begin{lem}\label{Edge proj} 
Let $N\geq 1$ and $\sigma\geq 0$. There exists
$R=R(N,\sigma)\geq 1$ such that the following holds:
Let $P:X\to T$ be a tree of proper geodesic spaces where all  edge spaces are
$\sigma$-contracting and uniformly $R$-separated. Let $e$ be a directed edge with terminal vertex $v$. Let  $x,y$ be two points in the edge space $Y^+_e$. Let $\gamma:[0,l]\to X$ be a 
$(N,0)$-quasi-geodesic with $\gamma(0)=x,\gamma(l)=y$
such that $P(\gamma)\cap e=\{v\}$ and for all $t\in (0,l)$, $\gamma(t)\not\in Y^+_e$. Then $\gamma$ does not intersect any other edge space other than $Y^+_e$.
\end{lem}
\begin{proof}
Let $C_N=5D(N,\sigma)$, where the constant $D(N,\sigma)$ is as in Theorem \ref{BQI}. Let $R=R(N,\sigma)=(2N+1)C_N+2ND_1+12D(N,\sigma) +1$, where 
$D_1=8D(1,\sigma)$ is the constant from Lemma \ref{Proj}.
If possible, let $\gamma$ intersect other edge spaces non-trivially. Let $e_1,...,e_m$ be the edges  incident on $v$ such that $\gamma$ intersects the edge spaces $Y^-_{e_1},...,Y^-_{e_m}$.
For each $Y^-_{e_i}$ intersected by $\gamma$, let $\gamma(s_i)$ be the first entry point and $\gamma(t_i)$ be the
last exit point. Replace the portion $\gamma|_{(s_i,t_i)}$ by a geodesic in $Y^-_{e_{i}}$ joining $\gamma(s_i)$ and $\gamma(t_i)$.
This results in a path $\gamma_v$ in $X_v$ joining $x$ and $y$.

  Consider the $C_N$-neighbourhood of $Y^+_e$ in $X_v$.
If possible, let there exists a maximal connected subsegment, $\gamma'_{v}$, of $\gamma_v$ which lie outside the $C_N$-neighbourhood of $Y^+_e$. Let $n$ be the number of edge spaces intersected by $\gamma'_{v}$.
Let $p$ and $q$ be the endpoints of $\gamma'_{v}$.
Since $d_X(Y^+_e,Y^-_{e_i})\geq R$, where $R>C_N>C_1$, so $p$ and $q$ will lie in $\gamma\cap X_v$.
As $\gamma'_{v}$ is the maximal connected subsegment of $\gamma_v$ lying outside the $C_N$-neighbourhood of $Y^+_e$, we have
$d_{X_v}(p,Y^+_e)=C_N=d_{X_v}(q,Y^+_e)$. As $Y^+_e$ is $\sigma$-contracting, by Lemma \ref{Proj} the diameter of projection of $Y^-_{e_i}$ into $Y^+_e$ is at most $D_1$ and the diameter of projection into $Y^+_e$ of a component of $\gamma\cap X_v$ lying outside the $C_N$-neighbourhood of $Y^+_e$ is at most $4D(N,\sigma)$ by Theorem \ref{GProj}.

Now projecting $\gamma'_{v}$ to $Y^+_e$, we get 
\begin{center}
    $(n-1)R\leq \ell_X(\gamma |_p^q)\leq N d_{X}(p,q)\leq Nd_{X_{v}}(p,q)\leq N(2C_N+nD_1+(n+1)(4D(N,\sigma)))$
\end{center}
Case (i) : Suppose $n=1$. Then $\ell_X(\gamma |_p^q)\leq N d_{X}(p,q)\leq Nd_{X_{v}}(p,q)\leq N(2C_N+D_1+8D(N,\sigma))$. 
So, the length of the portion of the quasi-geodesic $\gamma$ between $p$ and $\gamma(s_1)$ is at most $ N(2C_N+D_1+8D(N,\sigma))$ and hence $d_X(p,Y^-_{e_1})\leq N(2C_N+D_1+8D(N,\sigma))$.
This implies, $$R\leq d_X(Y^+_e,Y^-_{e_1})\leq d_X(p,Y^+_e)+d_X(p,Y^-_{e_1})\leq  C_N+N(2C_N+D_1+8D(N,\sigma))$$ Thus $R\leq (2N+1)C_N+ND_1+8D(N,\sigma)$ which is a contradiction.\\
Case (ii) : Suppose $n>1$. Then
\begin{eqnarray*}
   R&\leq&\dfrac{2NC_N+ND_1+8D(N,\sigma)}{n-1}+ND_1+4D(N,\sigma))\\&<&(2N+1)C_N+2ND_1+12D(N,\sigma) 
\end{eqnarray*}
which is a contradiction. \\
Hence, $\gamma_v\subset Nbhd({Y^+_e};C_N)\subset X_v$. As edge spaces are $R$-separated and $R>C_N$,
$\gamma_v$ does not intersect any other edge space other than $Y^+_e$. This also implies that, $\gamma_v=\gamma$ and hence $\gamma$ does not intersect any other incident edge spaces.
\end{proof}

\noindent {The underlying idea of Lemma \ref{Edge proj} is present in the work of  Szczepanski \cite{szczepanski}.}

\begin{thm}\label{main}
Let $N\geq 1$ and $\sigma\geq 0$. There exists
$R=R(N,\sigma)\geq 1$ such that the following holds:
Let $P:X\to T$ be a tree of proper geodesic spaces where all  edge spaces are
$\sigma$-contracting and uniformly $R$-separated. Let $\eta$ be a strongly contracting geodesic in a vertex space.
Then  every $(N,0)$-quasi-geodesic of $X$ with end points in $\eta$ lies in a bounded
neighbouhood of $\eta$ where the bound depends only on $N,\sigma$ and the contraction constant of $\eta$.

\end{thm}
\begin{proof}
Let $\eta$ be a $\rho$-contracting geodesic of a vertex space $X_v$ for some $\rho\geq0$. Let $x,y\in \eta$. 
Let $\alpha:[0,l]\to X$ be an arc length parameterization of a $(N,0)$-quasi-geodesic of $X$ with $\alpha(0)=x,\alpha(l)=y$. 
Without loss of generality,  we assume $\alpha(t)\not\in \eta$, for all $t\in (0,l)$, i.e. $\alpha$ does not intersect $\eta$ other than the end points. Then
$P(\alpha)$ will have diameter at most one by Lemma \ref{Edge proj}. 

Let $s(1),s(2),...,s(l)$ be the edges of $P(\alpha)$ incident on $v$. Let $v_i$ be another vertex of $s(i)$. If $\alpha_{i_k}$
is a portion of $\alpha$ such that the end points of $\alpha_{i_k}$ lie in $f_{s(i)}(X_{s(i)}\times\{1\})$ and $P(\alpha_{i_k})\cap s(i)=\{v_i\}$
then from Lemma \ref{Edge proj}, $\alpha_{i_k}$  does not intersect other edge spaces in $X_{v_i}$ and $\alpha_{i_k}$ lies in $X_{v_i}$.
Let $a_{i_k}, b_{i_k}$ be end points of $\alpha_{i_k}$. Then $\ell_{X_{v_i}}(\alpha_{i_k})=\ell_X(\alpha_{i_k})\leq Nd_X(a_{i_k},b_{i_k})\leq Nd_{X_{v_i}}(a_{i_k},b_{i_k})$. This implies  $\alpha_{i_k}$ is a $(N,0)$-quasi-geodesic in $X_{v_i}$. Similarly, if  $\alpha_{i_k}$
is a portion of $\alpha$ such that the end points of $\alpha_{i_k}$ lie in $f_{s(i)}(X_{s(i)}\times\{0\})$ and $P(\alpha_{i_k})\cap s(i)=\{v\}$ then $\alpha_{i_k}$ is a $(N,0)$-quasi-geodesic in $X_v$.
Let $\beta_i$ be the maximal subsegment of $\alpha$, containing all such $\alpha_{i_k}$'s, such that the end points of $\beta_i$ lie in $f_{s(i)}(X_{s(i)}\times\{0\})=Y^-_{s(i)}$. The Hausdorff distance between $Y^-_{s(i)}$ and  $Y^+_{s(i)}$ is at most one and as $Y^-_{s(i)},Y^+_{s(i)}$ are $\sigma$-contracting in respective vertex spaces, so by Lemma \ref{closeness} there exists $M>0$ depending on $N,\sigma$ such that $\beta_i$ lies in the $M+1$-neighbourhood of $Y^-_{s(i)}$. Let $M'=M+1$.
Let $p_i,q_i$ be end points of $\beta_i$. Then $p_i,q_i\in Y^-_{s(i)}\subset X_v$. Now following the proof of Lemma \ref{closeness} (ii), we divide $\beta_i$ into $\beta^1_i\cup\beta^2_i\cup...\cup \beta^{n(i)}_i$ such that the length of $\beta_i^j$ is one for all $j=1,2,...,n(i)-1$ and the length of $\beta_i^{n(i)}$ is at most one.  Let $x_i^j$ and $x_i^{j+1}$ be the end points of $\beta_i^j$.
Now for each $j$ there exists $y_i^j\in Y^-_{s(i)}$ such that $d_X\big(x_i^j,y_i^j\big)\leq M'$. 
Note that $x_i^1=y_i^1=p_i$ and $x_i^{n(i)+1}=y_i^{n(i)+1}=q_i$ and $d_X(y_i^j,y_i^{j+1})\leq 2M'+1$ for all $j$. Since $X_v$ is properly embedded in $X$, there exists $K'=K^{'}(M')\geq 1$ such that $d_{X_{v}}(y_i^j,y_i^{j+1})\leq K^{'}$, for all $j$. Successively joining $y_i^j$ and $y_{i}^{j+1}$  by a geodesic in $X_v$ for every $j$  we get a path $\Tilde{\beta_i}$ between $p_i$ and $q_i$ in $X_v$ such that $\beta_i$ and $\Tilde{\beta_i}$  lie in $M'+K'$-neighbourhood of each other. Now 
$$\ell_{X_{v}}(\Tilde{\beta_{i}}|_{y_i^{j}}^{y_i^{j'}})=\sum\limits_{r=j}^{j'} d_{X_v}(y_i^r,y_i^{r+1})\leq K'(j'-j)\leq
K^{'}\ell_X(\beta_{i}|_{x_i^{j}}^{x_i^{j'}})+K'.$$

\noindent For any $y_i,y'_i\in\Tilde{\beta_{i}}$ there exists $y^j_i,y^{j'}\in\Tilde{\beta_{i}}$ such that $\Tilde{\beta_{i}}|_{y_i}^{y'_i}\subset \Tilde{\beta_{i}}|_{y_i^{j}}^{y_i^{j'}}$, 
$d_{X_v}(y_i,y^j_i)\leq K'$ and $d_{X_v}(y'_i,y^{j'}_i)\leq K'$. Thus,
 $d_X(y_i,x^j_i)\leq M'+K'$, $d_X(y'_i,x^{j'}_i)\leq M'+K'$ and
$$\ell_{X_{v}}(\Tilde{\beta_{i}}|_{y_i}^{y'_i})\leq \ell_{X_{v}}(\Tilde{\beta_{i}}|_{y_i^{j}}^{y_i^{j'}})\leq K^{'}\ell_X(\beta_{i}|_{x^{j}_i}^{x^{j'}_i})+K^{'}.$$
The above process  results in a truncated path $\Tilde{\beta}$ in $X_v$ which is concatenation of $\Tilde{\beta_i}$'s and subsegments of $\alpha$, say $\xi_i$, which lies in $X_v$ connecting two consecutive $\Tilde{\beta_i}$'s (See Figure-1).
\begin{center}
\includegraphics[width=11cm,height=6cm]{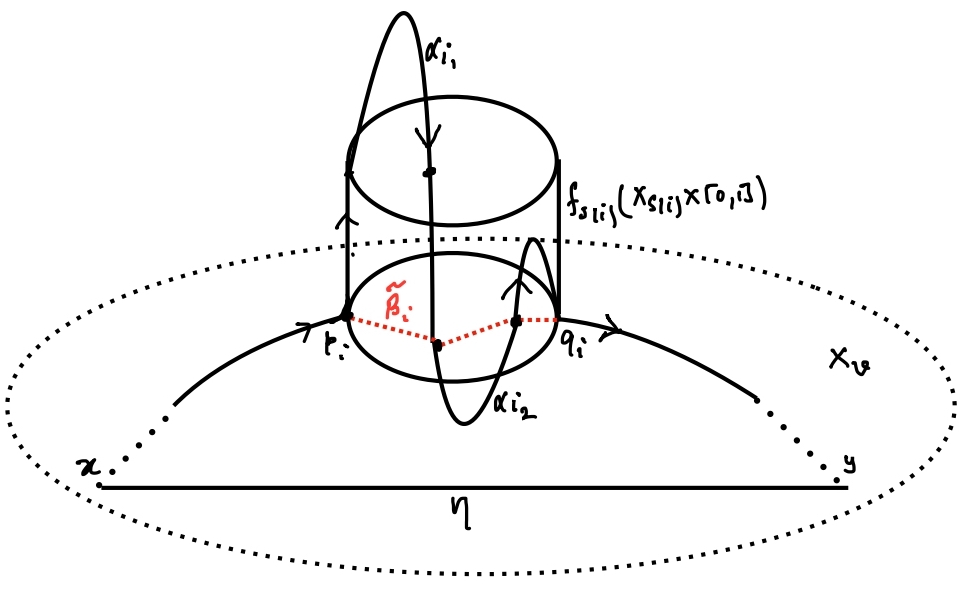}
\end{center}
\begin{center}
\underline{\small Figure - 1}
\end{center}
We claim that the truncated path $\Tilde{\beta}$ is quasi-geodesic in $X_v$. Let $c,d$ be two points in $\Tilde{\beta}$ and $\ell_{X_{v}}(\Tilde{\beta}|_{c}^{d})$ be the length of $\Tilde{\beta}$ between $c$ and $d$.\\
\begin{flushleft}
\underline{\textbf{Case 1:} $c\in \xi_k$ and $d\in \xi_m$ for some $k$ and $m$.}
\end{flushleft}

\begin{eqnarray*}
 \ell_{X_{v}}(\Tilde{\beta}|_{c}^{d})&=&\ell_{X_v}(\xi_k|_c^{p_k})+\sum\limits_{i=k+1}^{m-1}\ell_{X_{v}}(\xi_i)+\sum\limits_{i=k}^{m-1}\ell_{X_{v}}(\Tilde{\beta}_i)+ \ell_{X_v}(\xi_m|_{q_{m-1}}^d) \\
 &\leq& \ell_{X}(\xi_k|_c^{p_k})+\sum\limits_{i=k+1}^{m-1}\ell_{X}(\xi_i)+\sum\limits_{i=k}^{m-1}(K'\ell_X({\beta}_i)+K')+ \ell_{X}(\xi_m|_{q_{m-1}}^d)\\
 &\leq& K'\ell_X(\alpha|_{c}^{d})+(m-k)K'\\
 &\leq& K'\ell_X(\alpha|_{c}^{d})+K'\ell_X(\alpha|_{c}^{d})\\
 &\leq& 2NK'd_X(c,d) \leq 2NK'd_{X_v}(c,d)
\end{eqnarray*}

\begin{flushleft}
\underline{\textbf{Case 2:} $c\in\Tilde{\beta}_k$ and $d\in\Tilde{\beta}_m$ for some $k$ and $m$.}
\end{flushleft}
There exists $c'\in\beta_k$ and $d'\in\beta_m$ such that 
$d_X(c,c')\leq M'+K'$ and $d_X(d,d')\leq M'+K'$.

\begin{eqnarray*}
 \ell_{X_{v}}(\Tilde{\beta}|_{c}^{d})&=&\ell_{X_{v}}(\Tilde{\beta}_k|_{c}^{q_k})+ \sum\limits_{i=k+1}^{m}\ell_{X_{v}}(\xi_i)+\sum\limits_{i=k+1}^{m-1}\ell_{X_{v}}(\Tilde{\beta}_i)+\ell_{X_{v}}(\Tilde{\beta}_m|_{p_m}^{d})\\
 &\leq & K' \ell_X({\beta}_k|_{c'}^{q_k})+K'+\sum\limits_{i=k+1}^{m}\ell_X(\xi_i)+\sum\limits_{i=k+1}^{m-1}(K'\ell_X({\beta}_i)+K')\\&&+K'\ell_X({\beta}_m|_{p_m}^{d'})+K'\\
 &\leq& K'(\ell_X({\beta}_k|_{c'}^{q_k})+\sum\limits_{i=k+1}^{m}\ell_X(\xi_i)+\sum\limits_{i=k+1}^{m-1}\ell_X({\beta}_i)+\ell_X({\beta}_m|_{p_m}^{d'})) +(m-k)K'+2K'\\
 &\leq& K'\ell_X(\alpha |_{c'}^{d'}) +K'\ell_X(\alpha |_{c'}^{d'})+2K'\\
 &\leq& 2NK'd_X(c',d')+2K'\\
 &\leq& 2NK'(d_X(c,d)+2(M'+K'))+2K'\\
 &\leq&2NK'd_{X_v}(c,d)+4NK'(M'+K')+2K'.
\end{eqnarray*}
\begin{flushleft}
\underline{\textbf{Case 3:} $c\in\Tilde{\beta}_k$ and $d\in \xi_m$}.
\end{flushleft}
There exists $c'\in{\beta}_k$ such that $d_X(c,c')\leq M'+K'$.
\begin{eqnarray*}
 \ell_{X_{v}}(\Tilde{\beta}|_{c}^{d})&=&
 \ell_{X_{v}}(\Tilde{\beta}_i|_{c}^{q_k})+ \sum\limits_{i=k+1}^{m}\ell_{X_{v}}(\xi_i)+\sum\limits_{i=k+1}^{m-1}\ell_{X_{v}}(\Tilde{\beta}_i)+\ell_{X}(\xi_m|_{p_{m-1}}^{d})\\
 &\leq& K' \ell_X({\beta}_k|_{c'}^{q_k})+K'+ K'(\sum\limits_{i=k+1}^{m}\ell_X(\xi_i))+\sum\limits_{i=k+1}^{m-1}(K'\ell_X({\beta}_i)+K')+ K'\ell_{X}(\xi_m|_{p_{m-1}}^{d})\\
 &\leq& 2NK'd_X(c',d)+K'\\
 &\leq& 2NK'd_{X_v}(c,d)+2NK'(M'+K')+K'.
\end{eqnarray*}

Hence $\Tilde{\beta}$ is a $(2NK',4NK'(M'+K')+2K')$-quasi-geodesic in $X_v$, where $M'=M+1$.  Note that $K',M$ depends on $N,\sigma$.
As $\eta$ is $\rho$-contracting  there exists $\kappa=\kappa(N,\sigma,\rho)>0$ such that $\Tilde{\beta}\subset Nbhd(\eta;\kappa)$  (Contracting $\Rightarrow$ Morse, Theorem 1.3 of \cite{ACGH}). Also $\alpha\subset Nbhd(\Tilde\beta;M+1+K')$. Hence $\alpha\subset Nbhd(\eta;\kappa+M+1+K')$.
\end{proof}
By taking $N=1$ in Lemma \ref{Edge proj} and Theorem \ref{main}, we have the following corollary
\begin{cor}\label{main corollary}
Let $\sigma\geq 0$. There exists
$R=R(\sigma)\geq 1$ such that the following holds:
Let $X$ be a tree of proper geodesic spaces where all  edge spaces are
$\sigma$-contracting and uniformly $R$-separated.\\
(i)
Let $e$ be a directed edge with terminal vertex $v$. Let  $x,y$ be two points in the space $Y^+_e$. Let $\gamma:[0,l]\to X$ be a 
$X$-geodesic with $\gamma(0)=x,\gamma(l)=y$
such that $P(\gamma)\cap e=\{v\}$ and for all $t\in (0,l)$, $\gamma(t)\not\in Y^+_e$,
 Then the geodesic $\gamma$ does not intersect any other edge space other than $Y^+_e$.\\
(ii) Every strongly contracting geodesic  $\gamma$ in a vertex space is quasiconvex in $X$, where the quasiconvexity constant depends only on $\sigma$ and the contraction constant of $\gamma$.

\end{cor}
\begin{ex}
Let $G$ be a finitely generated group with a finitely generated subgroup $H$. Let $\Gamma_G$ denote the Cayley graph of $G$
with respect to some finite generating set.
Let $K$ be a coset of $H$ in $G$. We construct a one dimensional simplicial complex $\mathcal{H}(K)$  corresponding to $K$, it is called combinatorial horoball.
 
    \begin{itemize}
    \item $0$-skeleton of $\mathcal{H}(K)$, $\mathcal{H}{}(K)^{(0)}:= K^{(0)}\times (\{0,1,2,\cdots\}) $,
    \item $1$-skeleton of $\mathcal{H}(K)$,
    $\mathcal{H}{}(K)^{(1)}:=\big\{[(v,0),(w,0)] : v,w\in K^{(0)},
    [v,w]\in K^{(1)}\big\}\cup \big\{[(v,k),(w,k)] : v,w\in K^{(0)}, k>0, d_{K}(v,w)\leq 2^{k} \big\}\cup \big\{[(v,k),(v,k+1)] : v\in K^{(0)},
    \geq 0\big\}$.
    \end{itemize}
   Let  $\Gamma^{h}_{G}$ be the space obtained by gluing
    $\mathcal{H}(K)$ to $K$ in the Cayley graph $\Gamma_G$.
Suppose $G$ is hyperbolic relative to $H$. 
Then the augmented space $\Gamma^{h}_{G}$ is hyperbolic and $G$ acts properly discontinuously on $\Gamma^{h}_{G}$.

    We consider a subcomplex $\mathcal{H}_{m}^{c}(K)$ of $\mathcal{H}(K)$ whose $0$-skeleton is $K^{(0)}\times (\{m+1,m+2,\cdots\})$.
    Let $(\Gamma_{G})^{(m)}=\Gamma_{G}^{h}\backslash \mathcal{H}_{m}^{c}(K)$. 
    Then $G$ will act properly discontinuously  and co-compactly on $(\Gamma_{G})^{(m)}$ . Next we consider the subcomplex of $\Gamma_{G}^{h}$ 
    whose $0$-skeleton is $K^{(0)}\times (\{m\})$ and denote it by $\mathcal{H}_{m}(K)$. 
    It is known that peripheral subspaces are strongly contracting, for instance see Lemma 2.3 of \cite{sisto}. So,  $K$'s are strongly contracting in $\Gamma_G$.
    The Hausdorff distance between $\mathcal{H}_{m}(K)$ and $\mathcal{H}_{0}(K)=K$ is bounded by $m$. 
    Let $A$ be a  subset of diameter $r$ 
    in $\mathcal{H}_{0}(K)$ then the projection of $A$ into $\mathcal{H}_{m}(K)$ has diameter at most $\dfrac{r}{2^{m}}$ in $\mathcal{H}_{m}(K)$. Thus, $\mathcal{H}_{m}(K)$'s 
     are strongly contracting in $(\Gamma_{G})^{(m)}$ and if 
     the contraction constants for $\mathcal{H}_{0}(K),\mathcal{H}_{m}(K)$ are respectively $\kappa_0,\kappa_m$ then $\kappa_m\leq \kappa_0$. 
    Then we can choose large enough $m_{0}$ such that $\mathcal{H}_{m_{0}}(K)$ and 
    $\mathcal{H}_{m_{0}}(K')$ are $R(\kappa_0)$-separated for all distinct $K$ and $K^{'}$.  
    Let $T$ be the Bass-Serre tree of the graph of groups $G*_{H}G$. For every vertex of $T$ take a copy of $(\Gamma_{G})^{(m)}$ 
    and glue along appropriate translations of $\mathcal{H}_{m}(K)$ whenever there is an edge between two vertices. 
    So we get a tree of spaces, say $X$, where the conjugates of $G$
    in $G*_{H}G$ act properly discontinuously and co-compactly on the respective vertex spaces. 
    Let $\alpha$ be a strongly contracting geodesic in $\Gamma_G$. Since the Hausdorff distance between $(\Gamma_{G})^{(m)}$ and $\Gamma_G$ is at most $m$, $\alpha$ is strongly contracting in $(\Gamma_{G})^{(m)}$. As, $X$ satisfies all the assumptions of Corollary \ref{main corollary}, $\alpha$ is quasi-convex in $X$. Note that the group $G*_{H}G$ acts properly discontinuously and co-compactly on $X$, hence $G*_{H}G$ is quasi-isometric to $X$ (by \v{S}varc-Milnor Lemma). Thus, there exists $P\geq 1,\epsilon\geq 0$ such that 
    $\alpha$ is $(P,\epsilon)$-quasi-geodesic in $G*_{H}G$.

\end{ex}
\section{Applications}
Let $\delta\geq 0$. A geodesic triangle $\mathcal T$ is said to be $\delta$-\textit{slim} if each side of $\mathcal T$ is contained in the $\delta$-neighborhood of union of other two sides of $\mathcal T$. A geodesic metric space $X$ is said to be $\delta$-\textit{hyperbolic metric space} if all the triangles of $X$ are $\delta$-slim. A geodesic metric space is said to be \textit{hyperbolic} if it is $\delta$-hyperbolic for some $\delta\geq 0$.
Here we give a characterization of hyperbolic metric space, which we will used later.
\begin{prop}\label{bigon-char}
Let $Y$ be a geodesic metric space. $Y$ is hyperbolic if and only if there exists some $\delta>0$ such that all $(3,0)$-quasi geodesic bigons are $\delta$-slim (Here $\delta$-slimness of bigon means that the Hausdorff distance between the two sides is atmost  $\delta$).
\end{prop}
\begin{proof}
($\Rightarrow$) Suppose $Y$ is hyperbolic. Then from the stability of quasi-geodesics (Theorem 1.7, Pg.401, of \cite{bridson}), it follows that $(3,0)$-quasi-geodesics bigons are slim.\\
($\Leftarrow$) Suppose, there exists some $\delta>0$ such that all $(3,0)$-quasi geodesic bigons are $\delta$-slim. Let $ABC$ be an arbitrary triangle in $Y$ with sides $a$, $b$ and $c$ (sides $a$, $b$ and $c$ are opposite to the vertices $A$, $B$ and $C$ respectively). Let $D$ be a nearest point projection of vertex $A$ to the side $a$. Then $\alpha_1=\overline{AD}\cup \overline{DB}$ and $\alpha_2=\overline{AD}\cup \overline{DC}$ are $(3,0)$-quasi geodesics. So, we have two $(3,0)$-quasi geodesic bigons, namely $b\cup \alpha_1$ and $c\cup\alpha_2$. Now from our assumption, $b$ is in the $\delta$-neighbourhood of $\alpha_1$ and $\alpha_2$ is in the $\delta$-neighbourhood of $c$. Hence, $\delta$-neighbourhood of $a\cup c$ contains $\alpha_1\cup \alpha_2$ and so $2\delta$-neighbourhood of $a\cup c$ contains $b$. Similarly, $2\delta$-neighbourhood of $a\cup b$ contains $c$. \\
Again, taking the nearest point projection of the vertex $B$ to the side $b$, we can show that $a$ is contained in $2\delta$-neighbourhood of $b\cup c$. This proves that the triangle $ABC$ is $2\delta$-slim. Therefore, $Y$ is a hyperbolic metric space.
\end{proof}

The following theorem, due to Gromov, is a generalization of the fact that all triangles of a hyperbolic space are uniformly slim. Here we state a part of the version of the theorem given by Bestvina \& Feighn in \cite{best}.

\begin{thm}(See Section 3 of \cite{best} )\label{poly-resolution}(Resolution of Polygons of Hyperbolic Spaces) Let $\delta\geq 0$.
Let $Z$ be a $\delta$-hyperbolic metric space. Let $\Delta: D^2\rightarrow Z$ be a disk with boundary a $n$-sided $(K,\epsilon)$-quasi-geodesic polygon. Then there exist a function $B(\delta,n,K,\epsilon)$ (depending on $\delta$, $n$, $K$ and $\epsilon$ only) and a finite $\mathbb{R}$-tree $S$ with a map $r:D^2\rightarrow S$ such that:\\
(1) there are $n$ number of valence one vertices in $S$,\\
(2) $d_{Z}(\Delta(a),\Delta(b))\leq d_{S}(r(a),r(b))+B(\delta,n,K,\epsilon)$, for all $a,b\in S^1$,\\
(3) for all $s\in S$, $r^{-1}(s)$ is properly embedded finite tree  in $D^2$,\\
(4) if $e$ is an edge of $S$, then $r$ restricted to $r^{-1}(interior(e))$ is an $I$-bundle, where $I=[0,1]$.
\end{thm}
 Here note that, if two points $p=\Delta(a)$ and $q=\Delta(b)$ of $\Delta$ are identified in $S$, that is $r(a)=r(b)$, then the distance between them in $Z$ is at most $B(\delta,n,K,\epsilon)$. A fiber $r^{-1}(s)$ of  a  resolution is called a \textit{singular  fiber}  if $r^{-1}(s)$   is  not  isomorphic  to $I$.
\begin{thm} \label{hyp thm}Let $\delta,\sigma\geq 0$. Then there exists $\tilde R=\tilde R(\delta,\sigma)$ such that the following holds :

Let $\tilde{P}:\tilde{X}\rightarrow \tilde{T}$ be a tree of proper geodesic metric spaces. Suppose all the vertex spaces are $\delta$- hyperbolic, edge spaces are uniformly $\sigma$-contracting and $\tilde{R}$-separated then $\tilde{X}$ is hyperbolic.
\end{thm}
\begin{proof}

For a tree of hyperbolic metric spaces $P:X\to T$,
the vertex spaces are hyperbolic metric spaces and hence geodesics in a vertex space are uniformly strongly contracting (See Section 1.4 of \cite{ACGH}). Thus, by Theorem \ref{main}, if we take 
edge spaces to be $\sigma$-contracting and at least $R=R(3,\sigma)$-separated then
every $(3,0)$-quasi-geodesic of $X$ having end points in an edge space does not intersect other edge spaces. By Lemma \ref{closeness}, there exists $K_1\geq 1$ depending only on $\sigma$ such that the edge spaces are $(K_1,K_1)$-quasi-isometrically embedded in the respective vertex spaces. Also, there exists $K_2\geq 1$ and $\epsilon_2\geq 0$ depending only on $\sigma$ and not on the separation constant $R$ such that the portion of a $(3,0)$-quasi-geodesic of $X$ with end points in a vertex space of $P:X\to T$ gives rise to a truncated $(K_2,\epsilon_2)$-quasi-geodesic in that vertex space.
\par
Let $K=\max\{K_1,K_2,\epsilon_2,3\}$ and $\tilde{R}=\max\{R(3,\sigma), B(\delta,4,K,K)+1\}$, where $B(\delta,n,K,\epsilon)$ is the constant from Theorem \ref{poly-resolution}. Then $K$ depends only on $\sigma$ and  $\tilde R$ depends only on $\delta,\sigma$.
We will apply the characterization given in Proposition \ref{bigon-char}, to prove that $\tilde{X}$ is hyperbolic.
Let $x$, $y$ be any two points of $\tilde{X}$ and $\gamma_1$, $\gamma_2$ be two $(3,0)$-quasi geodesics between $x$ and $y$ in $\tilde{X}$. So, $\gamma_1\cup\gamma_2$ gives a $(3,0)$-quasi geodesic bigon in $\tilde{X}$. We will show that $\gamma_1$ and $\gamma_2$ lie in a bounded neighborhood of each other where the bound depends only on $\delta$ and $\sigma$. We  $\gamma_1,\gamma_2$ to be parameterized by arc-length.
\par
For an edge $e$, the map $f_e:\tX_e\times (0,1)\to \tX$ is an isometry onto its image. For the  convenience of notation, we denote the space $f_e(\tX_e\times (0,1))$ by $\tX_e\times (0,1)$.
Let $e_1$ be the edge with terminal vertex $v_1$ such that $x\in \tX_{e_1}\times (0,1)\sqcup \tX_{v_1}$ and
$e_{n+1}$ be the edge with initial vertex $v_n$ such that $y\in \tX_{v_n}\sqcup\tX_{e_{n+1}}\times (0,1)$. Let $e_1,v_1,e_{12},v_2,e_{23},v_3,...,v_n,e_{n+1}$ be the all vertices and edges in order lying on the geodesic between $\tilde{P}(x)$ and $\tilde{P}(y)$ in $\tilde{T}$. The paths $\tilde{P}(\gamma_1)$ and $\tilde{P}(\gamma_1)$ in $\tilde{T}$ will contain  $e_1,v_1,e_{12},v_2,e_{23},v_3,...,v_n,e_{n+1}$ of $T$, that means $\gamma_1$ and $\gamma_2$ will intersect the spaces $\tX_{e_1}\times (0,1) \sqcup\tX_{v_{1}},\tX_{e_{12}}\times (0,1),...,\tX_{v_n}\sqcup\tX_{e_{n+1}}\times (0,1)$.\par
For $i=1,2,...n$, let $p^{1}_{i}$ (respectively $p^{2}_{i}$) be the first entry point and $q^{1}_{i}$ (respectively $q^{2}_{i}$)  be the last exit point of $\gamma_1$ (respectively $\gamma_2$)  to $\tX_{v_i}$. If $x\in \tX_{v_1}$ then we take $p_1^1=p_1^2=x$ and if $y\in \tX_{v_n}$ then we take $q^1_1=q^2_1=y$.

By Theorem \ref{main}, the subsegments $\gamma_{1}\big\vert_{p^{1}_{i}}^{q^{1}_{i}}$ and  $\gamma_{2}\big\vert_{p^{2}_{i}}^{q^{2}_{i}}$ of $\gamma_1$ and $\gamma_2$ respectively lie in a uniformly bounded neighborhood of $\tX_{v_i}$. The truncated paths $\gamma^{1}_{i}$ and $\gamma^{2}_{i}$ in $\tX_{v_i}$ obtained from $\gamma_1$ and $\gamma_2$ respectively  are $(K_2,\epsilon_2)$ quasi geodesics and the Hausdorff distance between $\gamma^{1}_{i}$ (respectively $\gamma^{2}_{i}$) and $\gamma_{1}\big\vert_{p^{1}_{i}}^{q^{1}_{i}}$ (respectively $\gamma_{2}\big\vert_{p^{2}_{i}}^{q^{2}_{i}}$) is uniformly bounded by a number, say $L$, where $L$ depends on $\sigma$ only. We join $p^{1}_{i}$ and $p^{2}_{i}$ by a geodesic in $\tX_{e_{i-1,i}}\times\{1\}$ and call it $\beta_{i}^{1}$. Similarly we join $q^{1}_{i}$ and $q^{2}_{i}$ by a geodesic in $\tX_{e_{i,i+1}}\times\{0\}$ and call it $\beta_{i}^{2}$.
The edge spaces are $(K_1,K_1)$-quasi-isometrically embedded in the vertex spaces. Hence $\beta_{i}^{1}$ and $\beta_{i}^{2}$ are $(K_1,K_1)$-quasi-geodesics in $\tX_{v_i}$ for all $i$. We denote the subsegments of $\gamma_1$ (respectively $\gamma_{2}$) between ${q_{i-1}^{1}}$ and ${p_{i}^{1}}$ (respectively ${q_{i-1}^{2}}$ and ${p_{i}^{2}}$) by $\gamma^{1}_{i-1,i}$ (respectively $\gamma^{2}_{i-1,i}$). Note that $\gamma^{j}_{i-1,i}$ lies in $\tX_{e_{i-1,i}}\times [0,1]$, where $j=1,2$.


As $K=\max\{K_1,K_2,\epsilon_2,3\}$, for each $i=1,2,...n$, we have a disk $\Delta_i:D^2\rightarrow \tX_{v_i}$ with boundary a $({K},{K})$-quasi-geodesics $4$-gon (quadrilateral) in $\tX_{v_i}$, namely $\gamma^{1}_{i}\cup\beta_{i}^{2}\cup\gamma^{2}_{i}\cup\beta_{i}^{1}$. If $x\in \tX_{v_1}$ then the side $\beta^1_1$ is the  single point $x$ and if $y\in\tX_{v_n}$ then the side $\beta^2_n$ is the single point $y$.
Also, for $i=2,...,n$, we have disks $\Delta_{[i-1,i]}:D^2\rightarrow \tX_{e_{i-1,i}}\times [0,1]$ with boundary a $(K,K)$ quasi geodesics $4$-gon (quadrilateral) in $\tX_{e_{i-1,i}}\times [0,1]$, namely $\gamma^{1}_{i-1,i}\cup\beta_{i}^{1}\cup\gamma^{2}_{i-1,i}\cup\beta_{i-1}^{2}$. If $x\in \tX_{e_1}\times (0,1)$ (respectively $y\in \tX_{e_{n+1}}\times (0,1)$) we have a disk $\Delta_{e_1}:D^2\to \tX_{e_1}\times (0,1)$ 
with boundary a $(K,K)$-quasi-geodesic triangle $\gamma|_x^{p_1^1}\cup\beta^1_1\cup \gamma|_x^{p_1^2}$ and if $y\in \tX_{e_{n+1}}\times (0,1)$ we have a disk $\Delta_{e_{n+1}}:D^2\to \tX_{e_{n+1}}\times (0,1)$ with boundary a $(K,K)$-quasi-geodesic triangle $\gamma|_y^{q_n^1}\cup\beta^2_n\cup \gamma|_y^{q_n^2}$.\par

The initial and terminal disks containing $x,y$ respectively give rise to  quasi-geodesic triangles where we get only  single resolution for each disks by Theorem \ref{poly-resolution}. Apart from the initial and terminal disks, by Theorem \ref{poly-resolution},  we can have two types of resolutions $r_i:D^2\rightarrow S_i$, for each quadrilateral $\Delta_i$ (See Figure-2). 
In the first type, we will have two points $p\in \beta_{i}^{1}$, $q\in \beta_{i}^{2}$ which are identified in the $S_i$ and so $d_{\tX_{v_i}}(p,q)\leq B(\delta,4,K,K)$. And this contradicts the fact that edge spaces are $\tilde{R}$-separated (Note that $\tilde{R}> B(\delta,4,K,K)$). Therefore, we will only have the second type for $\Delta_i$'s. 
\begin{center}
\includegraphics[width=15cm,height=4cm]{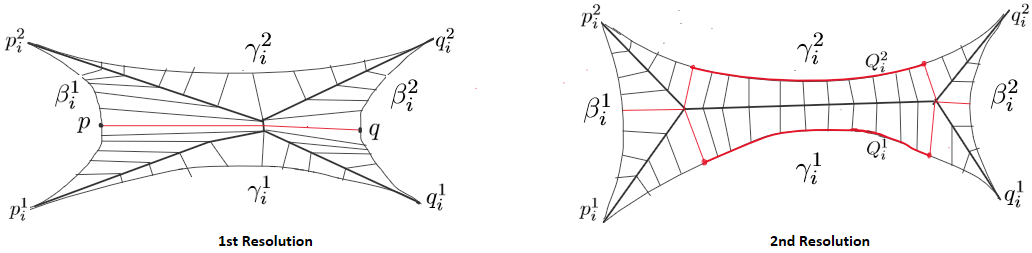}
\end{center}
\begin{center}
    \underline{\small Figure - 2}
\end{center}

In the second type, there are maximal segments, $Q^{1}_{i}$ and $Q^{2}_{i}$ of $\gamma^{1}_{i}$ and $\gamma^{2}_{i}$ respectively, which are identified pointwise in $S_i$. And so Hausdorff distance between those segments is at most $B(\delta,4,K,K)$. Let $u_{i}^{1}$ (respectively $u_{i}^{2}$) and $v_{i}^{1}$ (respectively $v_{i}^{2}$) be the end points of $Q^{1}_{i}$ (respectively $Q^{2}_{i}$). Note that $u_{i}^{1}$ and $u_{i}^{2}$ (respectively $v_{i}^{1}$ and $v_{i}^{2}$) are identified in $S_i$. Also there exist $w^{1}_{i}\in \beta_{i}^{1}$ such that $w^{1}_{i}$ is identified with $u_{i}^{1}$ and $u_{i}^{2}$ in $S_i$.

Similarly, we will have resolutions for $\Delta_{[i-1,i]}$'s. Here two types of resolution are possible. Now consider $\Delta_{[i-1,i]}$ and $\Delta_{i}$. Note that they have a common side, namely $\beta_{i}^{1}$. Depending on the type of resolution for $\Delta_{[i-1,i]}$, we will have to consider two cases.


\begin{center}
    \underline{\textbf{Case I}: $\Delta_{[i-1,i]}$ has type-2 resolution.}
\end{center}

Take any point $u\in \gamma_{i}^{1}\vert^{u^{1}_{i}}_{p^{1}_{i}}$, then there exists $w_{u}\in \beta_{i}^{1}$ such that $d_{\tX}(u,w_{u})\leq B(\delta,4,K,K)$. From the resolution of $\Delta_{[i-1,i]}$ there will exists $w_{u}'$ in either $\gamma_{i-1,i}^{1}$ or $\gamma_{i-1,i}^{2}$ such that $d_{\tX}(w_{u},w_{u}')\leq B(\delta,4,K,K)$. Hence, $d_{\tX}(u,w_{u}')\leq 2B(\delta,4,K,K)$.\\
Let $S_{[i-1,i]}$ be the finite tree and $r_{[i-1,i]}:D^2\to S_{[i-1,i]}$ be the map as in Theorem \ref{poly-resolution}
corresponding to the type-2 resolution of $\Delta_{[i-1,i]}$. There exists three points $s_{i-1,1}^1\in\gamma_{i-1,1}^1$,  $s_{i-1,1}^2\in\gamma_{i-1,1}^2$ and  $s_i^1\in\beta_i^1$ such that the set $\{s_{i-1,1}^1,s_{i-1,1}^2,s_i^1\}$ is mapped to a single vertex of $S_{[i-1,i]}$ under the map $r_{[i-1,i]}$ and diameter of the set $\{s_{i-1,1}^1,s_{i-1,1}^2,s_i^1\}$ is at most $B(\delta,4,K,K)$. Note that the points $u_i^1\in\gamma^1_i$, $u_i^2\in\gamma^2_i$ and $w_{u^1_i}\in\beta^1_i$ are mapped to a single vertex of the finite tree $S_i$ corresponding to the resolution of $\Delta_i$. Now two sub-cases arise.
\flushleft \underline{\textbf{Sub-case 1:}} Suppose $w_{u^1_i}$ comes before $s_i^1$ in $\beta_i^1$ (See Figure 3.1):\\
In this case $d_{\tX}(u_i^1,w_{u_i^1}')\leq 2B(\delta,4,K,K)$. The Hausdorff distance between $\gamma^{1}_{i}$ and $\gamma_1\vert_{p^{1}_{i}}^{q^{1}_{i}}$ is uniformly bounded by $L$. Hence, for each $u\in \gamma^{1}_{i}$, there exists  $\Bar{u}\in \gamma_1$ such that $d_{\tX}(\Bar{u},u)\leq L$.
Thus, there exists $\Bar{u}_i^1\in\gamma^1_i$ such that $d_{\tX}(u_i^1,\Bar{u}_i^1)\leq L$. By triangle inequality, $d_{\tX}(\Bar{u}_i^1,w_{u_i^1}')\leq 2B(\delta,4,K,K)+L$.
The path $\gamma_1$ is a $(3,0)$-quasi-geodesic. So, the length of the subsegment of $\gamma_1$ between $\Bar{u}_i^1$ and $w_{u_i^1}'$ is at most
$3\times (2B(\delta,4,K,K)+L)=6B(\delta,4,K,K)+3L$. The quasi-geodesic $\gamma_1$ passes through $p_i^1$, hence 
$d_{\tX}(u_i^1,p_i^1)\leq d_{\tX}(u_i^1,\Bar{u}_i^1)+d_{\tX}(\Bar{u}_i^1,p_i^1)\leq L+ 6B(\delta,4,K,K)+3L=6B(\delta,4,K,K)+4L.$ As vertex spaces are properly embedded in $\tX$, there exists a positive number $L'$ depending on $K$ such that $d_{\tX_{v_i}}(u_i^1,p_i^1)\leq L'$. The truncated path $\gamma_{i}^{1}\vert^{u^{1}_{i}}_{p^{1}_{i}}$ between ${u^1_i}$ and ${p^1_i}$ is a $(K,K)$-quasi-geodesic in $\tX_{v_i}$. Hence $\ell(\gamma_i^1|^{u^1_i}_{p^1_i})\leq KL'+K$ and  we have  $d_{\tX_{v_i}}(u,u^1_i)\leq KL'+K$, as $u\in \gamma_{i}^{1}\vert^{u^{1}_{i}}_{p^{1}_{i}}$. So, $d_{\tX}(u,u_i^2)\leq KL'+K+2B(\delta,4,K,K)$ and $d_{\tX}(u,\gamma_2)\leq KL'+K+2B(\delta,4,K,K)+L$.
\begin{center}
    \includegraphics[width=15cm,height=5.5cm]{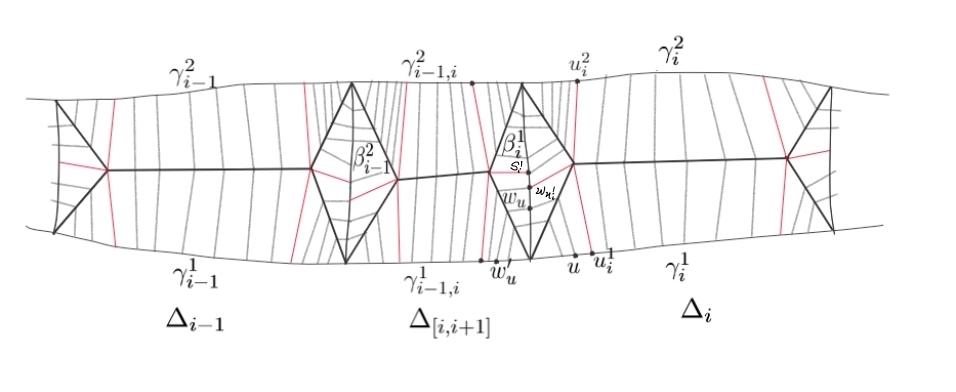}
    \underline{\small    Figure 3.1}
\end{center}

\flushleft \underline{\textbf{Sub-case 2:}}
Suppose $w_{u^1_i}$ comes after $s_i^1$ in $\beta_i^1$ (See Figure 3.2).\\
There exists $x_{i}\in \gamma_{i}^{1}\vert^{u^{1}_{i}}_{p^{1}_{i}}$ such that $s_i^1=w_{x_i}$. We take $x_{i}\in \gamma_{i}^{1}\vert^{u^{1}_{i}}_{p^{1}_{i}}$
to be the last point such that $s_i^1=w_{x_i}$. Now $w_{x_{i}}'\in\gamma_{i-1,i}^{1}$ and $d_{\tX}(x_i,w_{x_{i}}')\leq 2B(\delta,4,K,K)$.
The Hausdorff distance between $\gamma^{1}_{i}$ and $\gamma_1\vert_{p^{1}_{i}}^{q^{1}_{i}}$ is uniformly bounded by $L$. Hence, for each $u\in \gamma^{1}_{i}$, there exists  $\Bar{u}\in \gamma_1$ such that $d_{\tX}(\Bar{u},u)\leq L$. Therefore, from the fact that $\gamma_1$ is a $(3,0)$-quasi geodesic in $\tX$, we have
\begin{center}
$\ell_{\tX}(\gamma_1\vert^{\Bar{x}_{i}}_{p^{1}_{i}})\leq \ell_{\tX}(\gamma_1\vert^{\Bar{x}_{i}}_{w_{x_i}'})\leq 3.d_{\tX}(\Bar{x}_{i},w_{x_i}')\leq 3\Big(d_{\tX}(\Bar{x}_{i},x_i)+d_{\tX}(x_{i},w_{x_i}')\Big)$
\end{center}
\flushright$\leq 3L+6B(\delta,4,K,K)$ $\hspace{2.58cm}\cdots (\ast)$

\flushleft Then, $d_{\tX}(x_i,p^{1}_{i})\leq d_{\tX}(x_i,\Bar{x}_i)+d_{\tX}(\Bar{x}_i,p^{1}_{i})\leq L+\ell_{\tX}(\gamma_1\vert^{\Bar{x}_{i}}_{p^{1}_{i}})$ \flushright $\leq L+6B(\delta,4,K,K)+3L$ 
\hspace{2.4cm}\big( from $(\ast)$ \big)\\
\flushleft Since, vertex spaces are uniformly properly embedded in $X$, hence there exists $L'>0$ such that \hspace{.3cm}$d_{\tX_{v_i}}(x_i,p^{1}_{i})\leq L'$. \hspace{.3cm} Again, since $\gamma_{i}^{1}$ is a $(K,K)$-quasi geodesic in $\tX_{v_i}$, hence $\ell_{\tX}(\gamma_{i}^{1}\vert^{x_{i}}_{p^{1}_{i}})=\ell_{\tX_{v_{i}}}(\gamma_{i}^{1}\vert^{x_{i}}_{p^{1}_{i}})\leq KL'+K$. 


If $u\in \gamma_{i}^{1}\vert_{x_{i}}^{u^{1}_{i}}$, then $w_{u}'\in\gamma_{i-1,i}^{2}\subset \gamma_2$ and $d_{\tX}(u,w_{u}')\leq 2B(\delta,4,K,K)$.\\

Let $u\in \gamma_{i}^{1}\vert_{p^{1}_{i}}^{x_{i}}$. Now, $d_{\tX}(p_i^1,w_{x_i}')\leq\ell_{\tX}(\gamma_1\vert^{p_i^1}_{w_{x_i}'})\leq \ell_{\tX}(\gamma_1\vert^{\Bar{x}_{i}}_{w_{x_i}'})\leq 6B(\delta,4,K,K)+3L$. (from ($\ast$)). There exists $w_{x_i}''\in\gamma^2_{i-1,i}\subset\gamma_2$ such that $d_{\tX}(w_{x_i}',w_{x_i}'')\leq B(\delta,4,K,K)$. Therefore, 
\begin{eqnarray*}
 d_{\tX}(u,\gamma_2)\leq d_{\tX}(u,w_{x_i}'')
 &\leq& d_{\tX}(u,p_i^1)+d_{\tX}(p_i^1,w_{x_i}')+d_{\tX}(w_{x_i}',w_{x_i}'')\\
 &\leq&\ell_{\tX_{v_{i}}}(\gamma_{i}^{1}\vert^{x_{i}}_{p^{1}_{i}})+6B(\delta,4,K,K)+3L+B(\delta,4,K,K)\\
 &\leq& KL'+K+7B(\delta,4,K,K)+3L.
\end{eqnarray*}

\begin{center}
    \includegraphics[width=14cm,height=4.8cm]{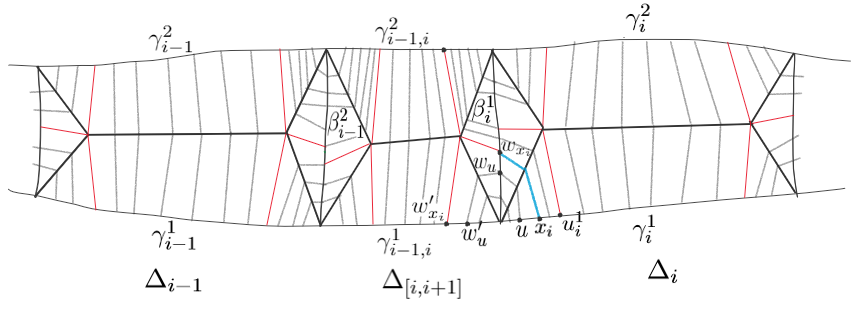}
\end{center}
\begin{center}
\underline{\small    Figure - 3.2}
\end{center}


\begin{center}
    \underline{\textbf{Case II}: $\Delta_{[i-1,i]}$ has type-1 resolution.}
\end{center}

Take any point $u\in \gamma_{i}^{1}\vert^{u^{1}_{i}}_{p^{1}_{i}}$, then there exists $w_{u}\in \beta_{i}^{1}$ such that $d_{\tX}(u,w_{u})\leq B(\delta,4,K,K)$. From the resolution of $\Delta_{[i-1,i]}$ there will exists $w_{u}'$ in either   $\gamma_{i-1,i}^{1}$ or $\gamma_{i-1,i}^{2}$ or $\beta_{i-1}^2$ such that $d_{\tX}(w_{u},w_{u}')\leq B(\delta,4,K,K)$. \\ 

\vspace{.2cm}
\underline{\textbf{Sub-case 1:}} Suppose $w_{u}'\in\beta_{i-1}^2$: $\beta_{i-1}^2$ is the common side of disks $\Delta_{[i-1,i]}$ and $\Delta_{i-1}$. We  consider the resolution of $\Delta_{i-1}$. There will exists $w_{u}''$ in either $\gamma_{i-1}^{1}$ or $\gamma_{i-1}^{2}$ such that $d_{\tX}(w_{u}',w_{u}'')\leq B(\delta,4,K,K)$ and so by triangle inequality, we have $d_{\tX}(u,w_{u}'')\leq 3B(\delta,4,K,K)$.
If $w_u''\in \gamma^2_{i-1}$ then we are done. Suppose $w_u''\in \gamma^1_{i-1}$.  There exists $\Bar u\in\gamma_1$ and $\Bar{w}_u''\in\gamma_1$ such that $d_{\tX}(u,\Bar u)\leq L$ and $d_{\tX}(w_u'',\Bar{w}_u'')\leq L$. The length of the subsegment of $(3,0)$-quasi-geodesic $\gamma_1$ between $\Bar u$ and $\Bar{w}_u''$ is at most $3\times (3B(\delta,4,K,K)+2L)$. The quasi-geodesic $\gamma_1$ passes through $p_i^1$ and $q_{i-1}^1$. Following the process as in \textbf{Case-I}, we get a number $L_1$ in terms of $B(\delta,4,K,K)$ and $L$ such that
$\ell(\gamma_i^1 |_{u}^{p_i^1})+\ell(\gamma^1_{i-1,i})+\ell(\gamma_{i-1}^1|_{q^1_{i-1}}^{w_u''})\leq L_1.$ 
If we take $w_u''$ to be the point on singular fiber then $d_{\tX}(w_u'',\gamma^2_{i-1})\leq B(\delta,4,K,K)$. In general, any path $\gamma_i^1 |_{u}^{p_i^1}*\gamma^1_{i-1,i}*\gamma_{i-1}^1|_{q^1_{i-1}}^{w_u''}$ is contained in a path of length at most $L_1$ and whose one end correspond to a singular fiber in the disk $\Delta_{i-1}$.
Hence, $d_{\tX}(u,\gamma^2_{i-1})\leq L_1+B(\delta,4,K,K)$ and $d_{\tX}(u,\gamma_2)\leq L_1+B(\delta,4,K,K)+L$ 
(See Figure-4).\\
\vspace{.2cm}
\underline{\textbf{Sub-case 2:}} $w_{u}'\in\gamma_{i-1,i}^{1}$ or $\gamma_{i-1,i}^{2}$. In this case, we will follow the exact same process as in \textbf{Case-I} to find a point in $\gamma_{i-1,i}^{2}\subset\gamma_2$ which is at uniformly bounded distance from $u$.\\
\begin{center}
    \includegraphics[width=14cm,height=4.8cm]{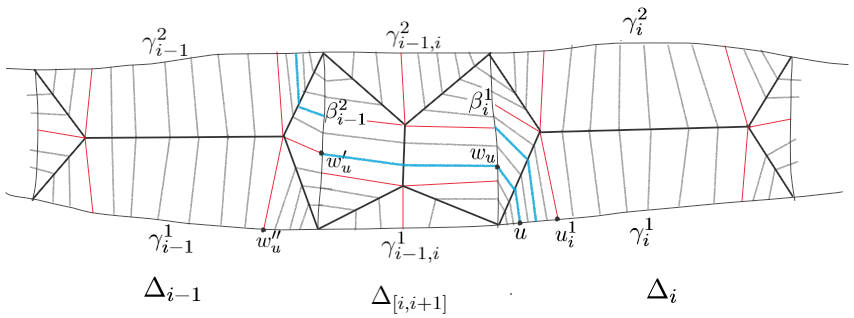}
\end{center}
\begin{center}
\underline{\small    Figure - 4}
\end{center}
If $u\in\gamma^1_{i-1,i}\subset\gamma_1$ then as above there exists a point $w_u''$ in $\gamma^2_{i-1}\cup\gamma^2_{i-1,i}\cup\gamma^2_i$, and hence a point in $\gamma_2$, which is at uniformly bounded distance from $u$.
For the initial and terminal disks, take the left (respectively right) end of Figure 3.1, Figure 3.2 and Figure 4 to be the point $x$ (respectively $y$)  and repeat the same argument as above. So, we get 
that each point of $\gamma_1$ lies in a bounded distance of $\gamma_2$ where the bound depends only on $\delta$ and $\sigma$. We can reverse the roles of $\gamma_1$ and $\gamma_2$ to conclude that 
they lie in a bounded neighbourhood of each other where the bound depends only on $\delta$ and $\sigma$. This implies $\tX$ is hyperbolic. As vertex spaces are $\delta$-hyperbolic geodesic spaces, the geodesics in a vertex space are uniformly strongly contracting and hence by Corollary \ref{main corollary}, vertex spaces are quasiconvex in $\tX$.

\end{proof}


\begin{thebibliography}{1}

\bibitem{ACGH}
Goulnara N. Arzhantseva, Christopher H. Cashen, Dominik Gruber, and David Hume,
\newblock{\it Characterizations  of  Morse  geodesics  via  superlinear  divergence  and  sublinear  contraction,}
\newblock{\it Doc. Math.22(2017), 1193-1224.}

\bibitem{best}
 M.Bestvina \& M.Feighn, 
\newblock{\it A Combination theorem for Negatively Curved groups, }
\newblock{   Jour.  Diff.  Geom.  ,  vol.   35, Pages 85-101,  1992.}

\bibitem{bridson}
M.Bridson \& A.Haefliger,
\newblock{ Metric Spaces of Non-positive curvature,} 
\newblock{ Volume 319, Springer.}

\bibitem{Charney15}
R. Charney and H. Sultan, \textit{Contracting boundaries of CAT(0) spaces},
J. Topol. 8 (2015),
no. 1, 93–117.

\bibitem{kapo}
Ilya Kapovich,
\newblock{\it The Combination Theorem and Quasiconvexity,}
\newblock{\it  Int. J. Algebra Comput., 11, 185 (2001)}

\bibitem{szczepanski}
A. Szczepanski, 
\newblock{\it Relatively hyperbolic groups},
\newblock{\it Michigan Math. Journal 45 (1998), 611-618.}


\bibitem{sisto}
Alessandro Sisto,
\newblock{\it Quasi-convexity of hyperbolically embedded subgroups},
\newblock{\it Math. Z. (2016) 283:649-658}

\end{thebibliography}
\end{document}